\documentclass[11pt]{amsart}

\usepackage{amsmath, amssymb}
\usepackage{amsthm}

\usepackage{color}

 \makeatletter
    
    \@addtoreset{equation}{section}
  \makeatother

\def\@begintheorem#1#2{%
   \trivlist
   \item[\hskip \labelsep{\bfseries #1\ #2.}]%
   \itshape}
\def\@opargbegintheorem#1#2#3{%
   \trivlist
   \item[\hskip \labelsep{\bfseries #1\ #2.\ (#3)}]%
   \itshape}
\makeatother

 \makeatletter
    
    \@addtoreset{equation}{section}
  \makeatother

\theoremstyle{plain}
\newtheorem{thm}{Theorem}[subsection]

\newtheorem{sthm}{Theorem}[section]
\newtheorem{lem}[thm]{Lemma}
\newtheorem{prop}[thm]{Proposition}
\newtheorem{cor}[thm]{Corollary}
\newtheorem{slem}[sthm]{Lemma}
\newtheorem{sprop}[sthm]{Proposition}

\theoremstyle{remark}
\newtheorem{rem}[thm]{Remark}
\newtheorem{srem}[sthm]{Remark}

\theoremstyle{definition}
\newtheorem{dfn}[thm]{Definition}
\newtheorem{sdfn}[sthm]{Definition}

\theoremstyle{plain}
\newtheorem{theor}{Theorem}

\newtheorem{nclaim}{Claim}

\newcommand{\QLS}{{\rm QLS}}
\newcommand{\B}{{\mathcal{B}}}
\newcommand{\QB}{{\mathcal{QB}}}

\newcommand{\QBG}{{\rm QBG}}
\newcommand{\qwt}{{\rm qwt}}
\newcommand{\ed}{{\rm end}}
\newcommand{\id}{{e}}

\newcommand{\gch}{{\rm gch}}
\newcommand{\wt}{{\rm wt}}
\newcommand{\Dg}{{\rm Deg}}
\newcommand{\Dgw}{{\rm Deg}^{w \lon}}
\newcommand{\dr}{{\rm dir}}

\newcommand{\aff}{{\rm aff}}

\newcommand{\ext}{{\rm ext}}

\newcommand{\eqdef}{:=}
\newcommand{\eqdefr}{=:}
\newcommand{\ardef}{\overset{\rm def}{\Leftrightarrow}}

\newcommand{\bqed}{\quad \hbox{\rule[-0.5pt]{6pt}{6pt}}  \vspace{3mm}}
\newcommand{\lon}{w_\circ}
\newcommand{\lons}{w_\circ (S)}

\newcommand{\tra}{t (\lambda_- ) }
\newcommand{\inv}{^{-1}}

\newcommand{\bp}{\mathbf{p}}
\newcommand{\bq}{\mathbf{q}}

\newcommand{\pair}[2]{\langle #1,\,#2 \rangle}

\newcommand{\Fg}{{\mathfrak{g}}}

\newcommand{\BZ}{\mathbb{Z}}

\newenvironment{enu}{%
 \begin{enumerate}%
}{\end{enumerate}}

\newcommand{\trr}{\widetilde{\Delta}_{\aff}}
\newcommand{\ptrr}{\widetilde{\Delta}_{\aff}^{+}}
\newcommand{\ntrr}{\widetilde{\Delta}_{\aff}^{-}}

\newcommand{\qad}{U_{\mathsf{v}}(\Fg_\aff)}
\newcommand{\qa}{U'_{\mathsf{v}}(\Fg_\aff)}
\newcommand{\da}{W_{\mathsf{v}}(\lambda)}

\title[Generalized Weyl modules and level-zero Demazure modules]{Generalized Weyl modules and Demazure submodules of level-zero extremal weight modules.}

\author[F.~Nomoto]{Fumihiko Nomoto}
\address[Fumihiko Nomoto]
 {Department of Mathematics, Tokyo Institute of Technology,
  2-12-1 Oh-Okayama, Meguro-ku, Tokyo 152-8551, Japan}
\email{nomoto.f.aa@m.titech.ac.jp}

\begin{document}
\maketitle
\begin{abstract}
We study a relationship between
the graded characters of 
generalized Weyl modules $W_{w \lambda}$, $w \in W$,
over the positive part of the affine Lie algebra and
those of specific quotients $V_{w}^- (\lambda) / X_{w}^- (\lambda)$, $w \in W$,
of the Demazure submodules $V_{w}^- (\lambda)$
of the extremal weight modules $V(\lambda)$ over the quantum affine algebra,
where $W$ is the finite Weyl group and $\lambda$ is a dominant weight.
\end{abstract}

\subjclass{{\scriptsize {\it Mathematics Subject Classification} 2010: Primary 17B37; Secondary 33D52, 81R10, 81R50.}}

%

\section{Introduction}
Let $\Fg$ be a finite-dimensional simple Lie algebra over $\mathbb{C}$,
and $\Fg_\aff$ the untwisted affine Lie algebra associated
to $\Fg$.
In \cite{OS}, Orr-Shimozono gave a formula for the 
specialization at $t=0$ of nonsymmetric Macdonald polynomials
in terms of quantum alcove paths;
they also gave a similar formula for the specialization at $t=\infty$.
Recently, for a dominant weight $\lambda$ for $\Fg$ and
an arbitrary element $w$ of the finite  Weyl group $W$,
Feigin-Makedonskyi \cite{FM} introduced a generalized Weyl module, denoted by $W_{w \lon \lambda}$, over
the positive part $\mathfrak{n}_\aff$
of the affine Lie algebra $\Fg_\aff$, and proved that its graded character
equals a certain graded character $C_{w}^{t(\lon \lambda)}$ of the set $\QB(w; t(\lon \lambda))$ of quantum alcove paths
with starting point $w$ and directions given by a reduced decomposition of $t(\lon \lambda)$,
where 
$\lon$ is the longest element of $W$, and
$t(\lon \lambda)$ is an element of the extended affine Weyl group
giving the translation by $\lon \lambda$.


Soon afterward, in \cite{NNS},
we proved that the graded character of the specific
quotient module $V_{w}^- (\lambda) / X_{w}^- (\lambda)$ of $V_{w}^- (\lambda)$
is identical to a certain graded character $\gch_w \QLS(\lambda)$ of the set $\QLS(\lambda)$ of
quantum Lakshmibai-Seshadri (QLS) paths of shape $\lambda$;
here, $V(\lambda)$ is the extremal weight module
of extremal weight $\lambda$
over the quantum affine algebra $\qad$,
and $V_w^- (\lambda) \subset V(\lambda)$ is the Demazure submodule over the negative part $U^{-}_{\mathsf{v}}(\Fg_\aff)$ of $\qad$.
In particular, in the case $w = \lon$,
this graded character was shown to be
the specialization $E_{\lon \lambda}(q, \infty)$
of the nonsymmetric Macdonald polynomial $E_{\lon \lambda}(q, t)$ at $t=\infty$.

Also, in \cite{LNSSS3}, it is proved that for an arbitrary
$w \in W$, the specialization $E_{w \lambda}(q, 0)$
at $t = 0$ is identical to the graded character of a specific quotient $U_{w}^{+}(\lambda)$ of the Demazure
submodule $V_{w}^{+}(\lambda) \subset V(\lambda)$.
Here we remark that for $w \in W$, $U_{w}^{+}(\lambda) \subset U_{\lon}^{+}(\lambda) = V_{\lon}^{+}(\lambda)/X_{\lon}^{+}(\lambda)$ in the notation of [NS2]; in contrast, there is no inclusion relation between $U_{\lon}^{+}(\lambda)$ and
$V_{w}^{+}(\lambda)/X_{w}^{+}(\lambda)$, or equivalently,
between $U_{e}^{-}(\lambda)$ and $V_{w}^{-}(\lambda)/X_{w}^{-}(\lambda)$.

The purpose of this paper is to reveal the relationship between
the graded character of the generalized Weyl module $W_{w\lambda}$ over $\mathfrak{n}_\aff$ and 
that of the quotient module $V_w^-(\lambda) / X_w^- (\lambda)$ of $V_w^-(\lambda)$ for an arbitrary $w \in W$.
More precisely,
we prove the following.
\begin{theor}[{=Theorem \ref{thm:eq_character}}]\label{intro:main_theorem}
Let $\lambda$ be a dominant weight, and $w \in W$.
Then, there holds the following equality{\rm:}
\begin{equation*}
\lon (\overline{\gch W_{w \lon \lambda}}) = \gch (V_{\lon w \lon}^-(\lambda) / X_{\lon w \lon}^-(\lambda)).
\end{equation*}
\end{theor}
\noindent
Here, $\overline{\ \cdot \textcolor{white}{i}}$ denotes the 
involution on $\mathbb{Q}(q)[P]$ given by:
$q \mapsto q^{-1}$,
$\overline{e^\mu} =e^\mu$ for $\mu \in P$,
and the $\mathbb{Q}(q)$-linear Weyl group action on $\mathbb{Q}(q)[P]$ is given by: $w \cdot e^{\mu} \eqdef e^{w\mu}$ for $w\in W$ and $\mu \in P$,
where $P$ is the weight lattice for $\Fg$.

In the case $w =\id$,
Lenart-Naito-Sagaki-Schilling-Shimozono
\cite{LNSSS2} constructed a bijection
$\Xi : \QB(e; t(\lon \lambda)) \rightarrow \QLS(\lambda)$
between quantum alcove paths and $\QLS$ paths
in order to prove that the graded character $\gch_{e}\QLS(\lambda)$ of $\QLS(\lambda)$
is identical to the specialization $E_{\lon \lambda}(q, 0)$ at $t=0$
of the nonsymmetric Macdonald polynomial $E_{\lon \lambda}(q, t)$.
In this paper, we generalize their construction to an arbitrary
$w \in W$, and prove Theorem \ref{intro:main_theorem}
by means of the constructed bijection.

This paper is organized as follows.
In Section 2, we recall some basic facts about the quantum Bruhat graph.
In Section 3,  we review the definitions of quantum alcove paths and
their graded characters. 
In Section 4, we recall the definition of QLS paths,
and then define some variants of their graded characters.
In Section 5, we give a bijection between the two sets $\QB(w; t(\lon \lambda))$
and $\QLS(\lambda)$ that preserves weights and degrees.
Using this bijection, we finally prove that $\lon (\overline{C_w^{t(\lon \lambda)}}) = \gch_{\lon w \lon}\QLS(\lambda)$,
and hence Theorem \ref{intro:main_theorem}.

\section{(Parabolic) Quantum Bruhat Graph}
Let $\mathfrak{g}$ be a finite-dimensional simple Lie algebra over $\mathbb{C}$,
$I$ the vertex set for the Dynkin diagram of  $\mathfrak{g}$,
and
$\{ \alpha_i \}_{i \in I }$
(resp., $\{ {\alpha}^{\lor}_i \}_{i \in I }$)
 the set of simple roots (resp., coroots) of  $\mathfrak{g}$.
Then
$\mathfrak{h} = \bigoplus_{i \in I}\mathbb{C}\alpha^\lor_i$ is a Cartan subalgebra of  $\mathfrak{g}$,
with
$\mathfrak{h}^* = \bigoplus_{i \in I}\mathbb{C}\alpha_i$ the dual space of $\mathfrak{h}$,
and $\mathfrak{h}_\mathbb{R}^* = \bigoplus_{i \in I}\mathbb{R}\alpha_i$ its real form;
the canonical pairing between  $\mathfrak{h}$ and  $\mathfrak{h}^*$ is denoted by
$\langle \cdot, \cdot \rangle : \mathfrak{h}^* \times \mathfrak{h} \rightarrow \mathbb{C}$.
Let $Q = \sum_{i \in I}\mathbb{Z}\alpha_i  
\subset \mathfrak{h}_\mathbb{R}^*$ denote 
the root lattice,
$Q^\lor = \sum_{i \in I}\mathbb{Z}\alpha_i^\lor  
\subset \mathfrak{h}_\mathbb{R}$ 
the coroot lattice,
and 
$P = \sum_{i \in I}\mathbb{Z}\varpi_i \subset \mathfrak{h}_\mathbb{R}^*$  the weight lattice of $\mathfrak{g}$,
where the $\varpi_i$, $i \in I$, are the fundamental weights for $\mathfrak{g}$,
i.e., 
$\langle \varpi_i ,\alpha_j^\lor \rangle = \delta_{i j}$
for $i, j \in I$;
we set $P^+ \eqdef \sum_{i \in I} \BZ_{\geq 0} \varpi_i$, and call an element $\lambda$ of $P^+$ a dominant (integral) weight.
Let us denote by $\Delta$ the set of roots,
and by $\Delta^{+}$ (resp., $\Delta^{-}$) the set of  positive (resp., negative) roots.
Also, let $W \eqdef \langle s_i \ | \ i \in I \rangle$
be the Weyl group of $\mathfrak{g}$,
where
$s_i $, $i \in I$, are the simple reflections acting on $\mathfrak{h}^*$ and on $\mathfrak{h}$
as follows:
\begin{align*}
s_i \nu = \nu - \langle  \nu , \alpha^\lor_i  \rangle \alpha_i, & \ \ \
\nu \in \mathfrak{h}^*,\\
s_i h = h - \langle \alpha_i , h \rangle \alpha^\lor_i,
& \ \ \
h \in \mathfrak{h};
\end{align*}
we denote the identity element and the longest element of $W$ by $e$ and $\lon$, respectively.
If $\alpha \in \Delta$ is written as $\alpha = w \alpha_i$ for 
$w\in W$ and $i \in I$, 
then its coroot $\alpha^\lor$ is $w \alpha^\lor_i$;
we often identify $s_\alpha$ with $s_{\alpha^\lor}$.
For  $u \in W$,
the length of $u$ is denoted by $\ell(u)$,
 which coincides with
 the cardinality of the set
$\Delta^+ \cap u^{-1}\Delta^-$.

\begin{sdfn}[\normalfont{\cite[Definition 6.1]{BFP}}]\label{QBG}
The quantum Bruhat graph, denoted by $\QBG (W)$, is the directed graph with vertex set $W$ whose directed edges are 
 labeled by positive roots as follows.
For $u,v \in W$, and $\beta \in \Delta^+$, 
an arrow $u \xrightarrow{\beta} v $ is a directed edge of $\QBG (W)$
if the following hold{\rm:}

\begin{enu}
\item
$v=u s_\beta$, and

\item
either
(2a):
$\ell(v)=\ell(u)+1$ or
(2b):
$\ell(v)=\ell(u) - 2\langle \rho, \beta^\lor \rangle +1$, 
\end{enu}
where $\rho \eqdef \frac{1}{2} \sum_{\alpha \in \Delta^+}{\alpha}$.
A directed edge satisfying (2a) 
(resp., (2b))
is called a Bruhat (resp., quantum) edge.
\end{sdfn}

\begin{srem}
The quantum Bruhat graph defined above is a ``right-hand'' version, while 
the one defined in \cite{BFP} is a ``left-hand'' version.
Note that results of \cite{BFP} used in this paper (such as Proposition \ref{shellability}) are unaffected by this difference
(cf. \cite{Po}).
\end{srem}

For a directed edge $u \xrightarrow{\beta} v $ of $\QBG (W)$,
we set
\begin{equation*}
\wt (u \rightarrow v)
\eqdef
\left\{
\begin{array}{ll}
      0 & \mbox{if} \ u \xrightarrow{\beta} v \mbox{ is a Bruhat edge}, \\
      \beta^{\lor} &  \mbox{if} \ u \xrightarrow{\beta} v \mbox{ is a quantum edge}.
\end{array}
\right.
\end{equation*}
Also, for $u,v \in W$, we take a shortest directed path
$u=x_0 \xrightarrow{\gamma_1} x_1 \xrightarrow{\gamma_2}\cdots \xrightarrow{\gamma_r} x_r=v$ in $\QBG (W)$,
and set
\begin{equation*}
\wt (u\Rightarrow v) \eqdef \wt (x_0 \rightarrow x_1)+\cdots +\wt (x_{r-1} \rightarrow x_r) \in Q^\lor;
\end{equation*}
we know from \cite[Lemma 1\,(2),\,(3)]{Po} that
this definition does not depend on the choice of a shortest directed path from $u$ to $v$ in $\QBG (W)$.
For a dominant weight $\lambda \in P^+$, we set $\wt_\lambda (u \Rightarrow v ) \eqdef  \pair{\lambda}{\wt (u\Rightarrow v)}$,
and call it the $\lambda$-weight of a directed path from $u$ to $v$ in $\QBG (W)$.

\begin{slem}[\normalfont{\cite[Lemma 2.1.3]{NNS}}]\label{involution}
If $x \xrightarrow{\beta} y$ is a Bruhat $($resp., quantum$)$ edge of $\QBG (W)$, then $y\lon \xrightarrow{- \lon \beta} x \lon$ is also a Bruhat $($resp., quantum$)$ edge of $\QBG (W)$.
\end{slem}


Let $w \in W$. We take (and fix) reduced expressions $w = s_{i_{1}}\cdots s_{i_{p}}$ and $\lon w^{-1} =s_{i_{-q}}\cdots s_{i_{0}} $;
note that
$		\lon =s_{i_{-q}}\cdots s_{i_{0}}s_{i_{1}}\cdots s_{i{p}}$
is also a reduced expression for the longest element $\lon$.
Now, we set
	\begin{equation}\label{inversion_root}
		\beta_{k} \eqdef
	s_{i_{p}}\cdots s_{i_{k+1}} \alpha_{i_{k}},
\ \ \ 
-q \leq k \leq p;
	\end{equation}
we have $\{\beta_{-q}, \ldots , \beta_{0}, \ldots , \beta_{p} \}=\Delta^+$.
Then we define a total order
$\prec$ on $\Delta^+$ by:
	\begin{equation}\label{reflectionorder}
		\beta_{-q} \prec \beta_{-q+1} \prec \cdots \prec \beta_{p};
	\end{equation}
note that this total order is a reflection order;
see Remark \ref{rem:weak_reflection_order}.

\begin{sprop}[{\cite[Theorem 6.4]{BFP}}] \label{shellability}
	Let $u, v\in W$.

\begin{enu}
\item
There exists a unique directed path from $u$ to $v$ in $\QBG (W)$ 
for which the edge labels are strictly increasing {\rm(}resp., strictly decreasing{\rm)} in the total order $\prec$ above.

\item
The unique label-increasing {\rm(}resp., label-decreasing{\rm)} path
\begin{equation*}
u = u_0
\xrightarrow{\gamma_1}
u_1
\xrightarrow{\gamma_2}
\cdots
\xrightarrow{\gamma_r}
u_r=v
\end{equation*}
 from $u$ to $v$ in $\QBG (W)$
is a shortest directed path from $u$ to $v$.
Moreover, it is lexicographically minimal {\rm(}resp., lexicographically maximal{\rm)} among all shortest directed paths from $u$ to $v${\rm;}
that is, for an arbitrary shortest directed path
\begin{equation*}
u = u'_0
\xrightarrow{\gamma'_1}
u'_1
\xrightarrow{\gamma'_2}
\cdots
\xrightarrow{\gamma'_r}
u'_r=v
\end{equation*}
from $u$ to $v$ in $\QBG (W)$, there exists some $1 \leq j \leq r$
such that $\gamma_j \prec \gamma'_j$ {\rm(}resp., $\gamma_j \succ \gamma'_j${\rm)},
and $\gamma_k = \gamma'_k$ for $1 \leq k \leq j-1$.
\end{enu}
	\end{sprop}

For a subset $S \subset I$,
we set
$W_S \eqdef \langle s_i \ | \ i \in S \rangle$;
notice that $S$ may be the empty set $\emptyset$.
We denote the longest element of $W_S$ by $\lons$.
Also, we set $\Delta_S \eqdef Q_S \cap \Delta$,
where $Q_S \eqdef \sum_{i \in S} \mathbb{Z}\alpha_i$, 
and $\Delta_S^+ \eqdef \Delta_S \cap \Delta^+ $, 
$\Delta_S^- \eqdef \Delta_S \cap \Delta^- $.
%
For $w\in W$, we denote  by 
$\lfloor w \rfloor$ the minimal-length coset representative for the coset $w W_S$ in $W / W_S$, and
for a subset $X \subset W$, we set $\lfloor X \rfloor \eqdef \{ \lfloor w \rfloor \ | \ w \in X \} \subset W^S$,
where $W^S \eqdef \lfloor W \rfloor$ is the set of minimal-length coset representatives for the cosets in $W / W_S$.
%
\begin{sdfn}[{\cite[Section 4.3]{LNSSS1}}]\label{QBGS}
The parabolic quantum Bruhat graph, denoted by $\QBG (W^S)$, is the directed graph with vertex set $W^S$ whose directed edges
are labeled by 
positive roots in
$ \Delta^+ \setminus \Delta^+_S$ as follows.
For $u,v \in W^S$, 
and $\beta \in \Delta^+ \setminus \Delta^+_S$,
an arrow $u \xrightarrow{\beta} v $ is a directed edge of $\QBG (W^S)$
if the following hold{\rm :}
\begin{enu}
\item
$v=\lfloor u s_\beta \rfloor$ and 
\item
either
{\rm (2a):} 
$\ell(v)=\ell(u)+1$ or
{\rm (2b):} $\ell(v)=\ell(u) - 2\langle \rho - \rho_S, \beta^\lor \rangle +1$, 
\end{enu}
where $\rho_S =\frac{1}{2}\sum_{\alpha \in \Delta^+_S}\alpha$.
A directed edge satisfying {\rm (2a)} (resp., {\rm (2b)}) is called a Bruhat (resp., quantum) edge.
\end{sdfn}

For a directed edge $u \xrightarrow{\beta} v $ in $\QBG (W^S)$, we set
\begin{equation*}
\wt^S (u \rightarrow v)
\eqdef
\left\{
\begin{array}{ll}
      0 & \mbox{if} \ u \xrightarrow{\beta} v \mbox{ is a Bruhat edge}, \\
      \beta^{\lor} &  \mbox{if} \ u \xrightarrow{\beta} v \mbox{ is a quantum edge}.
\end{array}
\right.
\end{equation*}
Also, for $u, v \in W^S$, we take a shortest directed path
$\bp:u=x_0 \xrightarrow{\gamma_1} x_1 \xrightarrow{\gamma_2}\cdots \xrightarrow{\gamma_r} x_r=v$ in $\QBG (W^S)$
(such a directed path always exists by \cite[Lemma 6.12]{LNSSS1}),
and set
\begin{equation*}
\wt^S (\bp) \eqdef \wt^S (x_0 \rightarrow x_1)+\cdots +\wt^S (x_{r-1} \rightarrow x_r) \in Q^\lor;
\end{equation*}
we know from \cite[Proposition 8.1]{LNSSS1} that if $\bq$ is another shortest directed path from $u$ to $v$ in $\QBG (W^S)$,
then $\wt^S(\bp) - \wt^S(\bq) \in Q_S^\lor \eqdef \sum_{i \in S} \BZ_{\geq 0} \alpha^\lor_i$.

Now, for a dominant weight
$\lambda \in P^+$,
we set 
\begin{equation*}
S = S_\lambda \eqdef \{ i \in I \ | \  \langle \lambda , \alpha^{\lor}_i \rangle =0 \}
.
\end{equation*}
By the remark just above,
for $u,v \in W^S$,
the value $\pair{\lambda}{\wt^S(\bp)}$ does not depend on the choice of a shortest directed path $\bp$
from $u$ to $v$ in $\QBG (W^S)$;
this value is called the $\lambda$-weight of a directed path from $u$ to $v$ in $\QBG (W^S)$.
Moreover,
we know from \cite[Lemma 7.2]{LNSSS2}
that the value $\pair{\lambda}{\wt^S(\bp)}$ is equal to the value 
$\wt_\lambda (x \Rightarrow y) =  \pair{\lambda}{\wt (x \Rightarrow y) }$ for all $x \in uW_S$ and $y \in vW_S$.
In view of this fact, for $u, v \in W^S$, we also write $\wt_\lambda (u \Rightarrow v)$ for the value $\pair{\lambda}{\wt^S(\bp)}$
by abuse of notation;
hence in this notation, we have
\begin{equation}\label{eq:QBG_and_parabolic_QBG_weight}
\wt_\lambda(x\Rightarrow y)= \wt_\lambda(\lfloor x \rfloor \Rightarrow \lfloor y \rfloor)
\end{equation}
for all $x, y \in W$.

\begin{sdfn}[{\cite[Section 3.2]{LNSSS2}}]
Let $\lambda \in P^+$ be a dominant weight and 
$\sigma \in \mathbb{Q}\cap [0,1]$,
and set $S=S_\lambda$.
We denote by
$\QBG_{\sigma \lambda} (W)$ (resp., $\QBG_{\sigma \lambda} (W^S)$ ) 
 the subgraph of $\QBG (W)$ (resp., $\QBG (W^S)$)
with the same vertex set but having only the edges:
$u \xrightarrow{\beta} v$ with $\sigma \langle \lambda, \beta^{\lor}  \rangle \in \mathbb{Z}$.
\end{sdfn}

\begin{slem}[{\cite[Lemma 6.1]{LNSSS2}}]\label{8.1}
Let $\sigma \in \mathbb{Q}\cap [0,1]${\rm;} notice that $\sigma$ may be $1$.
If $u \xrightarrow{\beta} v$ is a directed edge of $\QBG_{\sigma \lambda} (W)$,
then
there exists a  directed path from $\lfloor u \rfloor$ to $\lfloor v \rfloor$ in $\QBG_{\sigma \lambda} (W^S)$.
\end{slem}

Also, for $u,v\in W$, 
let $\ell(u\Rightarrow v)$ denote
the length of a shortest directed path in $\QBG (W)$ from $u$ to $v$.
For $w\in W$,
following \cite{BFP},
we define the $w$-tilted Bruhat order $\leq_{w}$ on $W$ as follows:
for $u,v \in W$,
	\begin{equation*}
		u \leq_{w} v
		\ardef
		\ell(v\Rightarrow w) =\ell(v\Rightarrow u)+\ell(u\Rightarrow w);
	\end{equation*} 
the $w$-tilted Bruhat order on $W$ is a partial order
with the unique minimal element $w$.

\begin{slem}[{\cite[Theorem 7.1]{LNSSS1}}, {\cite[Lemma 6.5]{LNSSS2}}]\label{8.5}
Let $u, v\in W^S$, and $w \in W_S$.

\begin{enu}
\item
There exists a unique minimal element in the coset $v W_S$
in the $u w$-tilted Bruhat order $\leq_{u w}$.
We denote it by $\min(v W_S, \leq_{u w})$.

\item
There exists a unique  directed path from some $x \in v W_S$ to $u w$
in $\QBG (W)$ whose edge labels are increasing 
in the 
total order $\prec$ on $\Delta^+$
defined in $(\ref{reflectionorder})$,
and lie in $\Delta^+ \setminus \Delta^+_S$.
This path begins with $\min(v W_S, \leq_{u w})$.

\item
Let
$\sigma \in\mathbb{Q}\cap [0,1]$, 
and
$\lambda \in P^+$  a dominant weight.
If there exists a directed path from $v$ to $u$ in $\QBG_{\sigma \lambda} (W^S)$,
then the directed path in $(2)$ is one in  $\QBG_{\sigma \lambda} (W)$.
\end{enu}
\end{slem}

\section{Quantum Alcove paths and their graded characters}

In this section,
we recall 
from \cite[Sections 4 and 5]{OS} and \cite[Section 1]{FM}
the definition and some of the properties of
the graded characters of quantum alcove paths.


Let
$\widetilde{ \mathfrak{g}}$ denote the  finite-dimensional simple Lie algebra
whose root datum is dual to that of $\mathfrak{g}$; the set of simple roots is
$\{ {\alpha}^{\lor}_i \}_{i \in I } \subset \mathfrak{h}$,
and the set of simple coroots is
$\{ \alpha_i \}_{i \in I } \subset \mathfrak{h}^*$.
We denote the set of roots of $\widetilde{ \mathfrak{g}}$ by $\widetilde{\Delta} = \{ \alpha^\lor \ | \ \alpha \in \Delta \}$,
and 
the set of positive (resp., negative) roots by
$\widetilde{\Delta}^{+}$ (resp., $\widetilde{\Delta}^{-}$).

We consider the untwisted affinization of the root datum of $\widetilde{ \mathfrak{g}}$.
Let us denote
by $\trr$ 
the set of all real roots,
and 
by $\ptrr$
(resp., $\ntrr$)
the set of all positive (resp., negative) real roots.
Then we have
$\trr=
 \{ \alpha^\lor +a \widetilde{\delta} \ | \ \alpha \in \Delta , a \in \mathbb{Z} \}$,
with  $\widetilde{\delta}$ the (primitive) null root.
We set
$\alpha^{\lor}_0 \eqdef \widetilde{\delta} - \varphi^{\lor} $,
where $\varphi \in \Delta$ denotes the highest short root, 
and set
$ I_{\aff}\eqdef I\sqcup \{0 \}$.
Then, $\{ \alpha^{\lor}_i \}_{i\in  I_{\aff}}$ is the set of simple roots.
Also, for 
$\beta \in \mathfrak{h}\oplus \mathbb{C}\widetilde{\delta}$,
we define ${\deg}(\beta) \in \mathbb{C} $ and 
$\overline{\beta} \in \mathfrak{h}$
by:
\begin{equation}\label{eq:dfn_deg}
\beta = \overline{\beta} + {\deg}(\beta) \widetilde{\delta}.
\end{equation}

We denote the Weyl group of $\widetilde{ \mathfrak{g}}$ by $\widetilde{W}$;
we identify $\widetilde{W}$ and $W$
through the identification of the simple reflections of the same index $i \in I$.
For $\nu \in \mathfrak{h}^*$,
let $t(\nu)$ denote the translation in $\mathfrak{h}^*$: $ t(\nu) \gamma = \gamma + \nu$ for $\gamma \in \mathfrak{h}^*$.
The corresponding affine Weyl group and the extended affine Weyl group
are defined by
$\widetilde{W}_{\aff}\eqdef t(Q) \rtimes W $ and
$\widetilde{W}_{\ext} \eqdef t(P) \rtimes W $,
respectively.
Also, we define $s_0 : \mathfrak{h}^* \rightarrow \mathfrak{h}^*$ by $\nu \mapsto \nu -( \langle \nu, \varphi^\lor  \rangle -1)\varphi $.
Then, $\widetilde{W}_{\aff}=\langle s_i \ | \ i \in I_{\aff}\rangle$;
note that $s_0 = t(\varphi) s_\varphi $.
The extended affine Weyl group $\widetilde{W}_{\ext}$ acts on
$\mathfrak{h}\oplus \mathbb{C}\widetilde{\delta}$
as linear transformations,
and on 
$\mathfrak{h}^*$ 
as affine transformations:
for $v\in W$, $t(\nu) \in t(P)$,
\begin{eqnarray*}
v t(\nu)( \overline{\beta}+r\widetilde{\delta} )=v\overline{\beta}+(r-\langle \nu, \overline{\beta} \rangle ) \widetilde{\delta},
&
\overline{\beta} \in  \mathfrak{h} , r \in \mathbb{C},
\\
v t(\nu) \gamma = v \nu +v \gamma,
&
\gamma \in \mathfrak{h}^* .
\end{eqnarray*}

An element $u \in \widetilde{W}_{\ext} $ can be written as 
\begin{equation}\label{eq:dfn_wt}
u=t({{\wt}(u)}) \dr (u),
\end{equation}
with ${\wt}(u) \in P$ and $ {\dr}(u) \in W$,
according to the decomposition
$\widetilde{W}_{\ext} = t(P) \rtimes W $.
For $w \in \widetilde{W}_{\ext}$,
we denote the length of $w$ by 
$\ell (w) $,
which equals
$\#
\left(\ptrr
\cap
w^{-1}\ntrr \right)
$.
Also, we set 
$\Omega 
\eqdef
\{ w \in \widetilde{W}_{\ext}  \ | \ \ell(w)=0 \}$.

Let $\mu \in P$, and
denote by $m_{\mu} \in \widetilde{W}_{\ext}$ the shortest element in the coset	$t(\mu)W$.
We take a reduced expression $m_{\mu} = u s_{\ell_{1}}\cdots s_{\ell_{L}}
\in \widetilde{W}_{\ext} = \Omega \ltimes \widetilde{W}_{\aff}$,
where $u \in \Omega$ and $ \ell_1 , \ldots , \ell_L \in  I_{\aff}$.
\begin{srem}[{\cite[Section 2.4]{M}}]
For a dominant weight $\lambda \in P^+$,
we have
\begin{equation}\label{equ:mcr_antidominant}
m_{\lon \lambda} = t(\lon \lambda).
\end{equation}
\end{srem}
\vspace{5mm}
Let $w \in W$.
For each $J = \{ j_{1} < j_{2} < j_{3} < \cdots < j_{r} \} \subset \{1,\ldots,L\}$,
we define an alcove path $p_{J} =
			\left( w m_\mu = z_0, z_{1} , \ldots , z_{r} ; \tilde{\beta}_{j_1} , \ldots , \tilde{\beta}_{j_r} \right)$ as follows: 
we set
$\tilde{\beta}_{k} \eqdef s_{\ell_{L}}\cdots s_{\ell_{k+1}} \alpha^{\lor}_{\ell_{k}} \in \ptrr$ 
for $1 \leq k \leq L$, and set
	\begin{eqnarray*}
		z_{0}&=&w m_{\mu} ,\\
		z_{1}&=&w m_{\mu}  s_{\tilde{\beta}_{j_{1}}},\\
		z_{2}&=&w m_{\mu} s_{\tilde{\beta}_{j_{1}}}s_{\tilde{\beta}_{j_2}},\\
				&\vdots& \\
		z_{r}&=&w m_{\mu} s_{\tilde{\beta}_{j_{1}}}\cdots s_{\tilde{\beta}_{j_r}}.
	\end{eqnarray*}
Also, following \cite[Section 3.3]{OS}, 
we set
	$\B (w;m_{\mu})
	\eqdef
	\left\{ p_{J} \ \left| \ J \subset \{ 1,\ldots ,L \}  \right. \right\}$
and
	$\ed (p_{J}) \eqdef z_{r}\in \widetilde{W}_{\ext}$.
Then we define
	$\QB(w; m_{\mu})$
to be the following subset of $\B (w; m_{\mu})$:
	\begin{equation*}
		\left\{
			 p_{J} \in \B (w ;m_{\mu}) \ 
			\left|  \ 
				\dr(z_{i}) 
				\xrightarrow{-\left(\overline{ \tilde{\beta}_{j_{i+1}}  } \right)^{\lor}}
				\dr (z_{i+1}) \ 
				\text{ is a directed edge of }\QBG (W), \ 0\leq  i \leq r-1
			\right. 
		\right\} ;
	\end{equation*}
an element of $\QB(w ; m_\mu)$ is called a quantum
alcove path with starting point $w$ and directions given by
a reduced decomposition of $m_\mu$.

	\begin{srem}[\normalfont{\cite[(2.4.7)]{M}}]
		If $j \in \{ 1, \ldots, L \}$,
		then $-\left(\overline{ \tilde{\beta}_{j}  } \right)^{\lor} \in {\Delta}^{+}$.
	\end{srem}

	For $p_{J} \in \QB(w ; m_{\mu})$, we define $\qwt (p_{J})$ as follows.
Let $J^- \subset J$ denote the set of those indices $j_i \in J$ for which
	$\dr(z_{i-1}) 
	\xrightarrow{-\left(\overline{\tilde{\beta}_{j_i}  } \right)^{\lor}} 
	\dr (z_{i})$ is a quantum edge of $\QBG(W)$.
Then we set
	\begin{equation*}
		{\qwt}(p_{J}) 
		\eqdef
		\sum_{j \in J^-} \tilde{\beta}_{j}.
	\end{equation*}

Now, following \cite[Definition 1.9]{FM},
let  $C_w^{m_{\mu}}$ 
denote the graded character
\begin{equation}
\sum_{p_{J} \in \QB(w ; m_{\mu}) } 
		q^{\deg({\qwt}(p_{J}))}e^{\wt(\ed (p_{J}))}
\end{equation}
of $\QB(w;m_\mu)$.

\begin{srem}[{\cite{OS}; see also \cite{FM}}]
For $\mu \in P$, we denote by $E_\mu (q,t)$ the nonsymmetric Macdonald polynomial, and 
by $E_\mu(q, 0)$ (resp., $E_\mu(q, \infty)$) 
its specialization $\lim_{t\rightarrow 0}E_\mu (q,t)$
(resp., $\lim_{t\rightarrow \infty}E_\mu (q,t)$) at $t=0$ (resp., $t=\infty$).
\begin{enu}
\item
For the special case $w=e$, 
it holds that 
\begin{equation*}
E_{\mu}(q, 0)=C_e^{m_{\mu}}.
\end{equation*}
\item
For the special case $w=\lon$, 
it holds that 
\begin{equation*}
E_{\mu}(q\inv, \infty)=\lon C_{\lon}^{m_{\mu}};
\end{equation*}
namely, 
\begin{equation*}
E_{\mu}(q\inv, \infty)=\sum_{p_{J} \in \QB(\lon ; m_{\mu}) } 
		q^{\deg({\qwt}(p_{J}))}e^{\lon \wt(\ed (p_{J}))}.
\end{equation*}
\end{enu}
\end{srem}
Let $\Fg_\aff$ denote the affine Lie algebra associated to $\Fg$, and let
$\Fg_\aff
=
\mathfrak{n}_\aff
\oplus
\mathfrak{h}_\aff
\oplus
\mathfrak{n}_\aff^-
$ be its triangular decomposition.
\begin{srem}
We should warn the reader that the root datum of the affine Lie algebra $\Fg_\aff$ is not necessarily dual to that of the untwisted affine Lie algebra
associated to $\widetilde{\Fg}$
whose set of real roots is $\widetilde{\Delta}_\aff$, though the root datum of $\widetilde{\Fg}$ is dual to that of $\Fg$.
In particular, for the index $0 \in I_\aff$, the simple coroot $\alpha_0^\lor = c - \theta^\lor$, where $\theta \in \Delta^+$ is the highest root of $\Fg$
and $c$ is the canonical central element of $\Fg_\aff$,
does not agree with the simple root $\alpha^\lor_0 = \widetilde{\delta}-\varphi^\lor$ 
(see the beginning of this section).
\end{srem}
%
\begin{sdfn}[{\cite[Definition 2.1]{FM}}]
Let $\lambda$ be a dominant weight, and $w \in W$.
Then the generalized Weyl module $W_{w \lon \lambda}$
is the cyclic $\mathfrak{n}_\aff$-module with
a generator $v$ and following relations:
\begin{align*}
(h \otimes t^k) v &= 0 \ \mbox{for all} \ h \in \mathfrak{h}, k>0,\\
(f_{\alpha}\otimes t)v &= 0\ \mbox{for}\ \alpha \in w \Delta^- \cap \Delta^-, \\
(e_{\alpha}\otimes 1)v &= 0\ \mbox{for}\  \alpha \in w \Delta^- \cap \Delta^+, \\
(f_{w \alpha}\otimes t)^{-\pair{\lon \lambda}{\alpha^\lor}}v &= 0\ \mbox{for}\  \alpha \in \Delta^+ \cap w^{-1}\Delta^-, \\
(e_{w \alpha}\otimes 1)^{-\pair{\lon \lambda}{\alpha^\lor}} v &= 0\ \mbox{for}\  \alpha \in \Delta^+ \cap w^{-1}\Delta^+, 
\end{align*}
where $e_\alpha$, $f_{-\alpha}$, $\alpha \in \Delta^+$, denote the Chevalley generators. 
\end{sdfn}
We can regard the generalized Weyl module $W_{w \lon \lambda}$ as $\mathfrak{n}_\aff \oplus \mathfrak{h}$-module by:
$h v = \pair{\lon\lambda}{h}v$, $h \in \mathfrak{h}$;
hence the module $W_{w \lon \lambda}$ is a $\mathfrak{h}$-weighted module.
Also, the module $W_{w \lon \lambda}$ has a grading defined by the conditions:
$\deg(v) = 0$, and the operator of the form $x \otimes t^k \in \mathfrak{n}_\aff$ increases a degree by $k$.
Thus, the graded character of the generalized Weyl module $W_{w \lon \lambda}$ is defined by:
\begin{equation*}
\gch \, W_{w \lon \lambda}
\eqdef
\sum
\dim \left(W_{w \lon \lambda}[\gamma, k]\right)q^k e^\gamma ,
\end{equation*}
where $W_{w \lon \lambda}[\gamma, k]$
denotes the subspace of $W_{w \lon \lambda}$ of degree $k$ and $\mathfrak{h}$-weight $\gamma$.
Feigin-Makedonskyi proved that
the graded character $\gch \, W_{w\lon\lambda}$ of $W_{w\lon\lambda}$
is identical to the graded character
$C_w^{t(\lon \lambda)}(=C_w^{m_{\lon \lambda}})$ (see \cite[Theorem 2.21]{FM} for details).

\section{Quantum Lakshmibai-Seshadri paths and some variants of their graded characters}
\subsection{Quantum Lakshmibai-Seshadri paths}
\begin{dfn}[{\cite[Definition 3.1]{LNSSS2}}]\label{def_qls}
Let $\lambda \in P^+$ be a dominant  weight,
and set $S \eqdef S_\lambda = \{ i \in I \ | \ \langle \lambda , \alpha^\lor_i \rangle =0 \}$.
A pair $\eta = (w_1, w_2 ,\ldots ,w_s ; \sigma_0, \sigma_1 , \ldots , \sigma_s ) $
of a sequence
$w_1,\ldots , w_s$ of  elements in $W^S$ 
such that $w_k \neq w_{k+1}$ for $1  \leq k \leq s-1$
   and an increasing sequence  
$0=\sigma_0< \cdots < \sigma_s=1$ of rational numbers,
is called a quantum Lakshmibai-Seshadri ($\QLS$) path of shape $\lambda$
if 
\begin{enu}
\item[(C)]
for every $1\leq i \leq s-1$, there exists a directed  path from $w_{i+1}$ to $w_{i}$ in $\QBG_{\sigma_i \lambda} (W^S)$.
\end{enu}
\noindent
Let $\QLS(\lambda)$ denote the set of all $\QLS$ paths of shape $\lambda$.
\end{dfn}

\begin{rem}\label{def_qls_rem}
As in \cite[Definition 3.2.2 and Theorem 4.1.1]{LNSSS4},
condition (C) can be replaced by{\rm:}
\begin{enu}
\item[(C')]
for every $1\leq i \leq s-1$, there exists a shortest directed  path in $\QBG (W^S)$ from $w_{i+1}$ to $w_{i}$
that is also a directed path in $\QBG_{\sigma_i \lambda} (W^S)$.
\end{enu}
\end{rem}

The set $\QLS(\lambda)$ provides a realization of the crystal basis of a particular quantum Weyl module $\da$
over $\qa$,
where
$\qa$ denotes
the quantum affine algebra without the degree operator. 
 (see \cite[Theorem 4.1.1]{LNSSS4}, \cite[Theorem 3.2]{NS-IMRN2005}, \cite[Remark 2.15]{Nakajima}).
Moreover,
$\QLS(\lambda) \cong \bigotimes_{i \in I} \QLS(\varpi_i)^{\otimes m_i}$ as $\qa$-crystals, where
$\lambda = \sum_{i \in I}m_i \varpi_i$;
in particular,
$\# \QLS(\lambda) = \prod_{i \in I} \left(\# \QLS(\varpi_i)\right)^{m_i}$.

\subsection{Some variants of graded characters}

Let $\lambda = \sum_{i \in I} m_i \varpi_i \in P^+$, $m_i \in \mathbb{Z}_{\geq 0}$, be a dominant weight, and $w \in W$.
For $\eta = (w_1, \ldots, w_s; \sigma_0, \ldots, \sigma_s) \in \QLS (\lambda)$, we set
\begin{align*}
\wt(\eta) &\eqdef \sum_{i=1}^{s} (\sigma_{i}- \sigma_{i-1})w_{i}\lambda \in P, \\
\Dg^* (\eta) &\eqdef \sum_{i=1}^{s-1} (1 - \sigma_i) \wt_\lambda (w_{i+1} \Rightarrow w_{i}), \\
\Dg_* (\eta) &\eqdef \sum_{i=1}^{s-1} \sigma_i \wt_\lambda (w_{i+1} \Rightarrow w_{i}),\\
\Dg^{w} (\eta) &\eqdef \Dg^* (\eta) + \wt_\lambda (w_{1} \Rightarrow w_{0}) =\sum_{i=0}^{s-1} (1 - \sigma_i) \wt_\lambda (w_{i+1} \Rightarrow w_{i}),\\
\Dg_{w} (\eta) &\eqdef \Dg_* (\eta) + \wt_\lambda (w_{s+1} \Rightarrow w_{s}) =\sum_{i=1}^{s} \sigma_i \wt_\lambda (w_{i+1} \Rightarrow w_{i}),
\end{align*}
where we set
$w_{0}\eqdef w$ and $w_{s+1}\eqdef w$.
Note that by Remark \ref{def_qls_rem}, 
we have $\sigma_{i} \wt_{\lambda}(w_{i+1} \Rightarrow w_{i})\in \mathbb{Z}_{\geq 0}$ for $1\leq i \leq s-1$.
Hence it follows that $\Dg^* (\eta), \Dg_* (\eta), \Dg^w (\eta),\Dg_w (\eta) \in \mathbb{Z}_{\geq 0}$;
notice that $\sigma_0 =0$, and $\sigma_s=1$.

Also, we set
\begin{align*}
\gch^w \QLS(\lambda) \eqdef \sum_{\eta \in \QLS (\lambda)} q^{-\Dg^w (\eta)} e^{\wt(\eta)}, \\
\gch_w \QLS(\lambda) \eqdef \sum_{\eta \in \QLS (\lambda)} q^{-\Dg_w (\eta)} e^{\wt(\eta)}.
\end{align*}

Let
$V(\lambda)$ be the extremal weight module of extremal weight $\lambda$ over the quantum affine algebra $\qad$, and by $V_w^-(\lambda) \subset V(\lambda)$
the Demazure submodule over the negative part $U^{-}_{\mathsf{v}}(\Fg_\aff)$ of $\qad$;
see \cite{NS-Z} for details.

In \cite{NNS},
we proved that 
$\gch V_w^-(\lambda)$ of the Demazure submodule $V_w^-(\lambda)$ of $V(\lambda)$
is identical to
\begin{equation*}
\left(\prod_{i \in I}\prod_{r=1}^{m_i}(1-q^{-r})\right)^{-1}\gch_w \QLS(\lambda).
\end{equation*}
Moreover, the graded character of a specific (finite-dimensional) quotient module $V_w^- (\lambda) / X_w^- (\lambda)$ of $V_w^- (\lambda)$
over 
$U^{-}_{\mathsf{v}}(\Fg_\aff)$
is identical to $\gch_w \QLS(\lambda)${\rm;}
see \cite[(5.7)]{NNS} for the definition of the quotient
module $V^-_w (\lambda) / X^-_w (\lambda)$.

\begin{rem}
Let $\lambda \in P^+$ be a dominant weight.
Then we know the following:
\begin{enu}
\item
if $w=e$, then $\gch_{e} \QLS(\lambda)=E_{\lon \lambda}(q^{-1}, 0)$ (\cite[Lemma 7.7 and Theorem 7.9]{LNSSS2});

\item
if $w=\lon$, then $\gch_{\lon} \QLS(\lambda)=E_{\lon \lambda}(q, \infty)$ (\cite[Theorem 3.2.7]{NNS}),
\end{enu}
where $E_{\lon \lambda}(q, 0)$ and $E_{\lon \lambda}(q, \infty)$ denote the specializations of the nonsymmetric Macdonald polynomial 
$E_{\lon \lambda}(q, t)$ at $t=0$ and $t=\infty$, respectively.
Here we remark that 
for $\eta \in \QLS(\lambda)$,
$\Dg^{*}(\eta)$
and
$\Dg_{w}(\eta)$
are
equal to 
$-\Dg (\eta)$ in \cite[Theorem 4.6]{LNSSS2}
and
$-{\mathrm{deg}}_{w \lambda} (\eta)$ in \cite[Section 3.2]{NNS},
respecively.
\end{rem}

\subsection{Lusztig involution}
Let $\lambda \in P^+$ be a dominant weight, and set
$S \eqdef S_\lambda = \{ i \in I \ | \ \langle \lambda , \alpha^\lor_i \rangle =0 \}$.
In this subsection, we state the relationship between the graded characters $\gch_w \QLS(\lambda)$ and $\gch^w \QLS(\lambda)$ for $w \in W$.

For $\eta = (w_1, \ldots, w_s; \sigma_0, \ldots , \sigma_s) \in \QLS(\lambda)$,
we define 
$S(\eta)$
to be $ (\lfloor \lon w_s \rfloor, \ldots , \lfloor \lon w_1 \rfloor; 1-\sigma_s, \ldots, 1-\sigma_0)$.
The following follows from \cite[Section 4.5]{LNSSS2}.
\begin{lem}\label{lem:luszting_involution}
\noindent
\begin{enu}
\item
Let $w_1, w_2\in W$. Then, $\wt_\lambda(\lon w_1\Rightarrow \lon w_2)=\wt_\lambda(w_2 \Rightarrow w_1)$.

\item
Let $\eta =  (w_1, \ldots, w_s; \sigma_0, \ldots, \sigma_s) \in \QLS(\lambda)$. Then we have
\begin{enu}
\item
$S(\eta)\in \QLS(\lambda)${\rm;}

\item
$\wt (S(\eta)) = \lon \wt (\eta)${\rm;}

\item
$\Dg_w(S(\eta)) = \Dg^{\lon w}(\eta)$.
\end{enu}
\end{enu}
\end{lem}
\begin{proof}
Parts (1), (2a), and (2b) are proved in \cite[Section 4.5]{LNSSS2}.
Let us prove part (2c).
It follows from \cite[Corollary 4.8]{LNSSS2} that
\begin{equation*}
\Dg_* (S(\eta)) = \Dg^*(\eta).
\end{equation*}
Also, by part (1), we have 
\begin{equation*}
\wt_\lambda( w \Rightarrow  \lfloor \lon w_{s} \rfloor )
=
\wt_\lambda( w \Rightarrow  \lon w_{s})=\wt_\lambda(w_s \Rightarrow \lon w)
\hspace{10mm}
\mbox{(by equation \eqref{eq:QBG_and_parabolic_QBG_weight})}
.
\end{equation*}
From these, we see that
\begin{align*}
\Dg_w(S(\eta)) 
&= \Dg_* (S(\eta)) + \wt_\lambda ( w \Rightarrow  \lon w_{s}) \\
&=\Dg^*(\eta)+ \wt_\lambda ( w_s \Rightarrow \lon w )
=\Dg^{\lon w }(\eta),
\end{align*}
as desired.
This proves the lemma.
\end{proof}
The operator $S$ above is an involution on $\QLS(\lambda)$,
 called the Lusztig involution.
By using Lemma \ref{lem:luszting_involution}, 
we deduce that
\begin{equation}\label{eq:lusztig_involution_graded_character}
\gch_w \QLS(\lambda) = \lon \left( \gch^{\lon w} \QLS(\lambda) \right) = 
 \sum_{\eta \in \QLS (\lambda)} q^{-\Dg^{\lon w} (\eta)} e^{\lon \wt(\eta)}
\end{equation}
for $w \in W$.

\section{Relationship between the two graded characters}
\subsection{Relationship between the graded characters $\gch \, W_{w\lambda}$ and $\gch (V^-_w(\lambda ) / X^-_w (\lambda))$ }
We define an involution $\overline{\ \cdot \textcolor{white}{i}}$ on $\mathbb{Q}(q)$ by $\overline{q}=q^{-1}$,
and set
$\overline{f} \eqdef \sum_{\mu \in P}\overline{f_\mu}e^\mu$
for $f=\sum_{\mu \in P}f_\mu e^\mu$ with $f_\mu \in \mathbb{Q}(q)$.
%
%
\begin{thm}\label{thm:graded_character}
Let $\lambda \in P^+$ be a dominant weight, and $w \in W$.
Then
we have
\begin{equation*}
\overline{C_{w}^{t(\lon \lambda)}} = \gch^{w \lon }\QLS(\lambda).
\end{equation*}
In particular, $\# \QB(w, t(\lon \lambda)) = \# \QLS(\lambda)$ {\rm(}see Proposition $\ref{bijective}${\rm)}.
\end{thm}
We will give a proof of Theorem \ref{thm:graded_character} in Section 5.3.
By combining Theorem \ref{thm:graded_character} and \eqref{eq:lusztig_involution_graded_character}
with $w$ replaced by $\lon w \lon$,
we obtain the following theorem.
\begin{thm}
Let $\lambda \in P^+$ be a dominant weight, and $w \in W$.
Then
we have
\begin{equation*}
\lon \left( \overline{C_{w}^{t(\lon \lambda)}} \right) = \gch_{\lon w \lon }\QLS(\lambda).
\end{equation*}
\end{thm}
Because
 $C_{w}^{t(\lon \lambda)}$ equals the graded character $\gch W_{w \lon \lambda}$ of the generalized Weyl module $W_{w \lon \lambda}$,
and $\gch_{\lon w \lon}\QLS(\lambda)$ equals the graded character $\gch (V^-_{\lon w \lon}(\lambda ) / X^-_{\lon w \lon} (\lambda))$ of the quotient module
$V^-_{\lon w \lon}(\lambda ) / X^-_{\lon w \lon} (\lambda)$,
we obtain the following.
\begin{thm}\label{thm:eq_character}
Let $\lambda \in P^+$ be a dominant weight, and $w \in W$. Then we have
\begin{equation}
\lon \left(\overline{\gch W_{w \lon \lambda}}\right) = \gch (V^-_{\lon w \lon}(\lambda ) / X^-_{\lon w \lon} (\lambda)).
\end{equation}
\end{thm}

\subsection{Reduced expressions for $t(\lon \lambda)$ and a total order on $\ptrr \cap t(\lon \lambda)\inv \ntrr $}
Let $\lambda \in  P^+ $ be a dominant weight,
and set $\lambda_- \eqdef \lon \lambda$, $S \eqdef S_\lambda =\{ i \in I \ | \ \langle \lambda , \alpha^\lor_i \rangle =0 \}$.
For $\mu \in W\lambda$, we denote by $v(\mu)\in W^S$ the minimal-length coset representative for the coset $\{ w \in W  \ | \ w \lambda = \mu \}$
in $W/W_S$;
note that
$\lon = v(\lambda_-) \lons$
and
$\ell(\lon) = \ell(v(\lambda_-)) + \ell(\lons)$,
where $\lons$ denotes the longest element of $W_S$.

In this subsection, we recall from \cite{NNS} a particular reduced expression for $m_{\lambda_-}$ $(=t({\lambda_-})$ by
\eqref{equ:mcr_antidominant})
with respect to a fixed total order on $\ptrr \cap t(\lon \lambda)\inv \ntrr $, and review some of its properties.

We fix reduced expressions
\begin{align}
v(\lambda_-)&=  s_{i_1}\cdots s_{i_M} \label{equ:vlambda}, \\ 
\lons &= s_{i_{M+1}}\cdots s_{i_N} \label{equ:lons}
\end{align}
for 
$v(\lambda_-)$ and $\lons$, respectively.
Then 
\begin{equation}\label{equ:lon}
  \lon =  s_{i_1}\cdots s_{i_N} 
  \end{equation}
is a reduced expression for $\lon$.
We set $\beta_j \eqdef s_{i_N} \cdots s_{i_{j+1}}\alpha_{i_j}$, $1 \leq j \leq N$.
Then we have
$\Delta^+ \setminus \Delta^+_S = \{ \beta_1, \ldots, \beta_M \}$ and
$\Delta^+_S  = \{ \beta_{M+1}, \ldots, \beta_N \}$.
We fix a total order on $\Delta^+$ such that
\begin{align}\label{weak_reflection_order}
\underbrace{ \beta_1 \succ \beta_2 \succ \cdots \succ \beta_M }_{\in  \Delta^+ \setminus \Delta^+_S} \succ 
\underbrace{ \beta_{M+1} \succ \cdots \succ \beta_{N}}_{\in \Delta^+_S} .
\end{align}

\begin{rem}\label{rem:weak_reflection_order}
We call the total order $\prec$ above a reflection order on $\Delta^+$;
if $\alpha ,\beta ,\gamma \in \Delta^+$ with $\gamma^\lor = \alpha^\lor + \beta^\lor$,
then $\alpha \prec \gamma \prec \beta$ or $\beta \prec \gamma \prec \alpha$.
\end{rem}

Now, we define an injective map
\begin{align*}
\Phi : \ptrr \cap \tra\inv \ntrr 
&\rightarrow
\mathbb{Q}_{\geq 0} \times \left( \Delta^+ \setminus \Delta^+_S \right) ,\\
\beta = \overline{\beta} + \deg(\beta) \widetilde{\delta}
&\mapsto 
\left(\frac{ \langle {\lambda_-} ,  \overline{\beta} \rangle -  \deg(\beta)}{ \langle {\lambda_-} ,  \overline{\beta} \rangle }  
,  \lon \overline{\beta}^\lor
\right).
\end{align*}
Here we note that $\langle {\lambda_-} ,  \overline{\beta} \rangle >0$, $\langle {\lambda_-} ,  \overline{\beta} \rangle -  \deg(\beta) \geq 0$,
and 
$\lon \overline{\beta}^\lor \in \Delta^+\setminus \Delta^+_S $
since $\langle {\lambda_-} ,  \overline{\beta} \rangle = \langle {\lambda} ,  \lon \overline{\beta} \rangle >0$;
recall from \cite[(2.4.7) (i)]{M}
that
\begin{equation}\label{B}
\ptrr \cap \tra\inv \ntrr =
\{ \alpha + a \widetilde{\delta} 
\ | \
\alpha \in \Delta^-,
a \in \mathbb{Z}, \mbox{ and }
0 < a \leq  \langle \lambda_- , \alpha^{\lor}
\rangle 
\}.
\end{equation}
Let us consider the lexicographic order $<$ on $\mathbb{Q}_{\geq 0} \times (\Delta^+ \setminus \Delta^+_S)$
induced by the usual total order on $\mathbb{Q}_{\geq 0}$
and
the reverse order of
the restriction to $\Delta^+ \setminus \Delta^+_S$
of the total order $\prec$ on $\Delta^+$ above;
that is, for $(a, \alpha), (b, \beta) \in \mathbb{Q}_{\geq 0} \times (\Delta^+ \setminus \Delta^+_S)$,
\begin{equation*}
(a, \alpha)<(b, \beta) \mbox{ if and only if } 
a<b, \mbox{ or }
a=b \mbox{ and } \alpha \succ \beta.
\end{equation*}
Then we denote by $\prec'$ 
the total order on $\ptrr \cap \tra\inv \ntrr$
induced by  the lexicographic order on $\mathbb{Q}_{\geq 0} \times (\Delta^+ \setminus \Delta^+_S)$
through the injective map $\Phi$.

The proof of the following proposition is the same as that of \cite[Proposition 3.1.8]{NNS}.
\begin{prop}\label{goodreducedexpression}
With the notation and setting above, let us write $\ptrr \cap \tra\inv \ntrr$ as 
$\left\{\gamma_1 \prec' \cdots \prec' \gamma_L \right\}$.
Then, there exists a unique reduced expression $\tra= u s_{\ell_1}\cdots s_{\ell_L}$ for $\tra$, with $u \in \Omega$ and  $\{\ell_1, \ldots , \ell_L \} \subset I_{\aff}$,
such that
$\tilde{\beta}_j 
=\gamma_j$
for $1 \leq j \leq L$, where 
$\tilde{\beta}_j = s_{\ell_L}\cdots s_{\ell_{j+1}}\alpha_{\ell_j}^\lor$, $1 \leq j \leq L$.
\end{prop}

In the following, we use the reduced expression $\tra= u s_{\ell_1}\cdots s_{\ell_L}$ for $\tra$ given by this proposition.
The proof of the following lemma is the same as that of \cite[Lemma 3.1.10]{NNS}.
\begin{lem}\label{lengthadditive}
Keep the notation and setting above.
Then, $u s_{\ell_{M+1}} \cdots s_{\ell_L}$ is a reduced expression for $m_\lambda$.
Moreover, if we write $us_{\ell_k} =s_{i'_k}u$ for $1 \leq k \leq M$,
then $i_k = i'_k$,
where $\lon = s_{i_1} \cdots s_{i_N}$
is the reduced expression \eqref{equ:lon} for $\lon$.
\end{lem}

We set 
$a_k \eqdef \deg(\tilde{\beta}_k) \in \mathbb{Z}_{> 0}$ for $1\leq k \leq L$;
since 
$\ptrr \cap \tra\inv \ntrr
= \{ \tilde{\beta}_1, \ldots , \tilde{\beta}_L \}$,
we see by (\ref{B}) that
$0<a_k \leq \langle {\lambda_-}, \overline{\tilde{\beta}_k} \rangle$.

\begin{cor}\label{akbk}
For $1 \leq k \leq M$, we have $\lon \overline{\tilde{\beta}_k} = \beta_k^\lor$,
where $\beta_k \eqdef s_{i_N} \cdots s_{i_{k+1}} \alpha_{i_k}$.
\end{cor}

\begin{proof}
If we set 
$\hat{\beta}_k \eqdef u s_{\ell_1}\cdots s_{\ell_{k-1}}\alpha_{\ell_k}^\lor$, $1 \leq k \leq M$,
then we have
\begin{align*}
-t({\lambda_-})\tilde{\beta}_k
&=
-( u s_{\ell_1}\cdots s_{\ell_L}) ( s_{\ell_L} \cdots s_{\ell_{k+1}} \alpha_{\ell_k}^\lor)
=
- u s_{\ell_1}\cdots s_{\ell_{k-1}}s_{\ell_k} \alpha_{\ell_k}^\lor \\
&=
- u s_{\ell_1}\cdots s_{\ell_{k-1}}( - \alpha_{\ell_k}^\lor)
=
u s_{\ell_1}\cdots s_{\ell_{k-1}} \alpha_{\ell_k}^\lor
=
\hat{\beta}_k.
\end{align*}
From this, together with the equality
$
-\tra \tilde{\beta}_k =
 -\overline{ \tilde{\beta}_k} 
 -(a_k - \langle {\lambda_-}, \overline{\tilde{\beta}_k} \rangle  )\widetilde{\delta}
$,
we obtain
$\overline{\hat{\beta}_k }=-\overline{\tilde{\beta}_k}$.
Therefore, we see that
\begin{align*}
\lon \overline{ \tilde{\beta}_k }=\lon (-\overline{\hat{\beta}_k})
&=
\lon(-\overline{ u s_{\ell_1}\cdots s_{\ell_{k-1}} \alpha_{\ell_k}^\lor})
=
\lon(-s_{i_1}\cdots s_{i_{k-1}} \alpha_{i_k}^\lor)
\hspace{10mm}
\mbox{(by \ Lemma~\ref{lengthadditive})}
\\
&=
(s_{i_N}\cdots s_{i_1})(-s_{i_1}\cdots s_{i_{k-1}}\alpha_{i_k}^\lor)=
s_{i_N}\cdots s_{i_{k+1}} \alpha_{i_k}^\lor
=
\beta_k^\lor,
\end{align*}
as desired.
\end{proof}

For $1 \leq k \leq L$, we set
	\begin{equation}\label{2.15}
		d_k \eqdef
			 \frac{ \langle \lambda_- ,  {\overline{\tilde{\beta}_k }} \rangle - a_k}{\langle \lambda_- ,  {\overline{\tilde{\beta}_k }} \rangle} ;
	\end{equation}
here, $d_k$ is just the first component of $\Phi(\tilde{\beta}_k) \in \mathbb{Q}_{\geq 0} \times (\Delta^+ \setminus \Delta^+_S)$.
Recall that for $1 \leq k,j \leq L$,
$\Phi(\tilde{\beta}_k)<\Phi(\tilde{\beta}_j)$ if and only if $k < j$.
Therefore, it follows that
	\begin{equation}\label{C}
		0\leq d_1 \leq \cdots \leq d_L \lneqq 1.
	\end{equation}
The following lemma follows from the definition of the map $\Phi$.
\begin{lem}[\normalfont{cf. \cite[Lemma 3.1.12]{NNS}}]\label{remark2.11}
If $1\leq k<j \leq L$ and $d_k =d_j$,
then	$\lon  \left(\overline{\tilde{\beta}_k }\right)^\lor  \succ \lon  \left(\overline{\tilde{\beta}_j }\right)^\lor $.
\end{lem}

\subsection{Proof of Theorem~\ref{thm:graded_character}}
We keep the notation of Section 5.2. 
In this subsection, 
in order to prove Theorem \ref{thm:graded_character},
we give a bijection
	\begin{equation*}
		\Xi_w : \QB(w; \tra) \rightarrow \QLS(\lambda)
	\end{equation*}that preserves  weights and degrees
for an arbitrary $w \in W$.
The way to construct this bijection
 is similar to the one for the bijection
$\Xi:\overleftarrow{\mathrm{QB}}(\id, \tra) \rightarrow \QLS(\lambda)$ in \cite[Section 3.3]{NNS}, where
$\overleftarrow{\rm QB}(\id, \tra)$ is a subset of $\B(\id; \tra)$ defined as follows:
\begin{equation*}
\{
p_J \in \B(\id; \tra) \ | \
\dr(z_{j_i}) \xleftarrow{-\left(\overline{\tilde{\beta}_{j_i}}\right)^\lor} \dr(z_{j_i+1}) \mbox{ is a directed edge of }\QBG(W) \mbox{ for }j \in J
\} .
\end{equation*}
Here we remark that the defining condition for $\overleftarrow{\rm QB}(\id, \tra)$ differs
from the one for $\QB(w, \tra)$ 
in that alllowable directed edges are reversed.
Also, note that we ignore the directed path \eqref{w'_p}
below for  for $p=0$
(which depends on $w \in W$)
in the construction of the map $\Xi_w$,
while we ignore the path (3.17) in \cite{NNS} for $p=0$
in the construction of the map $\Xi$.
Accodingly, in order to
 prove that $\Xi_w$ is both injective and surjective,
we need to check that for each $w \in W$ and $\eta \in \QLS(\lambda)$,
there exists a unique directed path in $\QBG(W)$ from some element in the coset $\kappa(\eta)W_S$ to $w\lon$
whose edge labels are increasing and lie in $\Delta^+\setminus \Delta^+_S$.

\begin{rem}\label{affinereflectionorder}
Let
$\gamma_1, \gamma_2 ,\ldots, \gamma_r \in \ptrr \cap \tra\inv \ntrr$,
and 
define a sequence
$\left( y_0 , y_1 , \ldots , y_r 
\right)$
 by:
$y_0 = w \tra $, and $y_i =y_{i-1}s_{\gamma_i}$ for $1 \leq i \leq r$.
Then, the pair of sequences 
$\left( y_0 , y_1 , \ldots , y_r ; \gamma_1 , \gamma_2, \ldots , \gamma_r \right)$
is an element of 
$\QB( w; \tra )$
if and only if the following conditions are satisfied{\rm :}

\begin{enu}
\item
$\gamma_1\prec' \gamma_2 \prec' \cdots \prec' \gamma_r$,
where the order $\prec'$ is 
the total order on $ \ptrr \cap \tra\inv \ntrr$
introduced in Section 5.2;

\item
for $1\leq i \leq r$,
$\dr(y_{i-1}) \xrightarrow{-\left( \overline{\gamma_i} \right)^\lor} \dr(y_{i})$
is
a directed edge of $\QBG (W)$.
\end{enu}
\end{rem}

Let us define a map
$\Xi_w : \QB (w ; \tra) \rightarrow \QLS (\lambda)$ as follows.
Let $p_{J}$ be an arbitrary element of  $\QB (w; \tra)$ of the form
	\begin{equation*}
		p_{J} = \left( w\tra = z_0 ,  z_{1} , \ldots , z_{r}
; \tilde{\beta}_{j_1} , \tilde{\beta}_{j_2}, \ldots , \tilde{\beta}_{j_r} \right) \in \QB (w; \tra ),
	\end{equation*}
where $J = \{ j_1 < \cdots < j_r\} \subset \{ 1 , \ldots , L \}$.
	We set $x_k \eqdef {\dr}(z_k)$,
	$0 \leq k \leq r$. Then, by the definition of $\QB (w ; \tra)$, 
	\begin{equation}\label{2.15.5}
		w =  x_0 \xrightarrow{- \left( \overline{ \tilde{\beta}_{j_1} } \right)^{\lor}  }   x_1
		\xrightarrow{- \left( \overline {\tilde{\beta}_{j_2} } \right)^{\lor} } \cdots  \xrightarrow{ - \left( \overline{ \tilde{\beta}_{j_r} } \right)^{\lor} }  x_r
	\end{equation}
	is a directed path in $\QBG (W)$; the equality
	$w=x_0$
	  follows from the equality $\dr(t(\lon \lambda)) = e$.
	We take $0 = u_0 \leq u_1 < \cdots < u_{s-1} < u_s=r$ and $0 = \sigma_0 < \sigma_1 < \cdots <\sigma_{s-1} < 1 = \sigma_{s}$ in such a way that
(see (\ref{C}))
\vspace{3mm}
	\begin{equation}\label{2.16}
		 \underbrace{ 0 = d_{j_1} = \cdots = d_{j_{u_1}} }_{=\sigma_0}
		< \underbrace{d_{j_{u_1 +1}} = \cdots =d_{j_{u_2}}}_{=\sigma_1} < \cdots <
		\underbrace{ d_{j_{u_{s-1}+1}} = \cdots =d_{j_r} }_{=\sigma_{s-1}} <1 = \sigma_{s};
	\end{equation}
	note that 
	$d_{j_1}>0$ if and only if $u_1=0$.
	We set
	$w'_p \eqdef x_{u_p}$ for $0 \leq p \leq s$. 
	Then, by taking a subsequence of (\ref{2.15.5}), we obtain the following  directed path in $\QBG (W)$
	for each $0 \leq p \leq s-1$:
	\begin{equation*}
		w'_p = x_{{u_p}} \xrightarrow{- \left( \overline{ \tilde{\beta}_{j_{u_p +1}} } \right)^{\lor} }   x_{{u_p +1}} 
		\xrightarrow{- \left( \overline {  \tilde{\beta}_{j_{u_p +2}} } \right)^{\lor} } \cdots  
		\xrightarrow{ -\left( \overline{  \tilde{\beta}_{j_{u_{p+1}}} } \right)^{\lor} }  x_{{u_{p+1}}} = w'_{p+1}.
	\end{equation*}
	Multiplying this directed path on the right by $\lon$, 
	we obtain the following directed path in $\QBG (W)$
	for each $0 \leq p \leq s-1$
	(see Lemma \ref{involution}):
	\begin{equation}\label{w'_p}
		w_p \eqdef
		w'_p \lon =
		x_{{u_p}} \lon \xleftarrow{ \lon \left( \overline{ \tilde{\beta}_{j_{u_p +1}} } \right)^{\lor} }  \cdots  
		\xleftarrow{ \lon \left( \overline{ \tilde{\beta}_{j_{u_{p+1}}} } \right)^{\lor} }  x_{{u_{p+1}}} \lon
		= w'_{p+1} \lon \eqdefr w_{p+1}.
	\end{equation}
	Note that $w_0 = w'_0 \lon = x_0 \lon = w \lon$. In addition, 
the edge labels of this directed path are increasing
in the reflection order $\prec$ on $\Delta^+$ given by \eqref{weak_reflection_order}
(see Lemma \ref{remark2.11}),
	  and lie in $\Delta^+ \setminus \Delta^+_S$;
	this property is used to prove that the map $\Xi_w$ is injective.
	Because
	\begin{equation}\label{sigma_with_a}
		\sigma_p \langle \lambda , \lon \overline{ {\tilde{\beta}_{j_u}} }  \rangle
		=
		d_{j_u} \langle \lambda , \lon \overline{ {\tilde{\beta}_{j_u}} }  \rangle
		=\frac{\langle \lambda_- ,  \overline{ {\tilde{\beta}_{j_u}} }  \rangle - a_{j_u}}{\langle \lambda_- ,  \overline{ {\tilde{\beta}_{j_u}} } \rangle} \langle \lambda , \lon \overline{ {\tilde{\beta}_{j_u}} }  \rangle
		=\langle \lambda_- ,  \overline{ {\tilde{\beta}_{j_u}} }  \rangle - a_{j_u} \in \mathbb{Z}
	\end{equation}
	for $u_p +1 \leq u \leq u_{p+1}$, $0 \leq p \leq s-1$,
	we find that (\ref{w'_p}) is a  directed  path in $\QBG_{\sigma_p \lambda} (W)$ for each $0 \leq p \leq s-1$.
	Therefore, by Lemma \ref{8.1}, there exists a  directed path in $\QBG_{\sigma_p \lambda}(W^S)$ from $\lfloor w_{p+1} \rfloor$ to $\lfloor w_p \rfloor$,
where
$S \eqdef S_\lambda = \{ i \in I \ | \ \langle \lambda , \alpha^\lor_i \rangle =0 \}$.
Also, we claim that $\lfloor w_p \rfloor \neq \lfloor w_{p+1} \rfloor$ for any $1 \leq p \leq s-1$. Suppose,
for a contradiction, that $\lfloor w_p \rfloor = \lfloor w_{p+1} \rfloor$
for some $p$.
Then, $w_p W_S = w_{p+1} W_S$,
and hence 
\begin{equation}\label{eq:min_tilt}\min(w_{p+1}W_S, \leq_{w_p}) = \min(w_{p}W_S, \leq_{w_p}) =w_p.
\end{equation}
Recall that 
the directed path (\ref{w'_p}) is a path in $\QBG (W)$ from $w_{p+1}$ to $w_{p}$ whose edge labels are increasing and lie in $\Delta^+ \setminus \Delta^+_S$.
By Lemma \ref{8.5}\,(1), (2), the directed path (\ref{w'_p}) is a shortest path in $\QBG (W)$ from $w_{p+1} = \min(w_{p+1}W_S, \leq_{w_p})$ to $w_p$.
It follows from \eqref{eq:min_tilt} 
that
$w_{p+1} = \min(w_{p+1}W_S, \leq_{w_p})
=w_p$,
and hence the length of the directed path (\ref{w'_p}) is equal to $0$.
Therefore, we have $\{ j_{u_p +1}, \ldots , j_{u_{p+1}} \} = \emptyset$, and hence $u_p = u_{p+1}$,
which contradicts the fact that $u_p < u_{p+1}$.
	Thus we obtain
	\begin{equation}\label{D}
		\eta \eqdef 
		(\lfloor w_{1} \rfloor , \ldots \ , \lfloor w_{s} \rfloor ; \sigma_0, \ldots, \sigma_s)  \in \QLS(\lambda).
	\end{equation}
We now define $\Xi_w (p_{J}) \eqdef \eta$.

In order to prove that the map $\Xi_w:\QB(w; \tra) \rightarrow \QLS(\lambda)$ is bijective,
we prove that the map is both injective and surjective;
in \cite[Sect. 3.3]{NNS}, we gave the inverse map $\Theta : \QLS(\lambda) \rightarrow \overleftarrow{\rm QB}(\id; \tra)$ of the map $\Xi$.

\begin{lem}\label{injective}
		The map $\Xi_w:\QB(w; \tra) \rightarrow \QLS(\lambda)$
is injective.
\end{lem}

\begin{proof}
Let $J= \{ j_1 , \ldots , j_r \}$ and $K = \{ k_1 , \ldots , k_{r'} \}$ be subsets of $\{ 1, \ldots , L\}$ such that
$\Xi_w (p_J)= \Xi_w (p_K) = (v_1, \ldots, v_s ; \sigma_0 , \ldots , \sigma_s) \in \QLS(\lambda)$.
As in (\ref{2.16}), we set $0 = u_0 \leq u_1 < \cdots < u_{s} = r$ and $0 = u'_0 \leq u'_1 < \cdots < u'_{s} = r'$
in such a way that
	\begin{align*}
		 \underbrace{0 = d_{j_1} = \cdots = d_{j_{u_1}} }_{=\sigma_0}
		< \underbrace{d_{j_{u_1 +1}} = \cdots =d_{j_{u_2}}}_{=\sigma_1} < \cdots <
		\underbrace{ d_{j_{u_{s-1}+1}} = \cdots =d_{j_r} }_{=\sigma_{s-1}} <1 = \sigma_{s},\\
		 \underbrace{0 = d_{k_1} = \cdots = d_{k_{u'_1}} }_{=\sigma_0}
		< \underbrace{d_{k_{u'_1 +1}} = \cdots =d_{k_{u'_2}}}_{=\sigma_1} < \cdots <
		\underbrace{ d_{k_{u'_{s-1}+1}} = \cdots =d_{k_{r'}} }_{=\sigma_{s-1}} <1 = \sigma_{s}.
	\end{align*}
As in (\ref{w'_p}), we consider the directed paths in $\QBG (W)$
	\begin{align}\label{surj_induction}
\begin{array}{ll}
		w_p \xleftarrow{ \lon \left( \overline{ \tilde{\beta}_{j_{u_p +1}} } \right)^{\lor} }  \cdots  
		\xleftarrow{ \lon \left( \overline{ \tilde{\beta}_{j_{u_{p+1}}} } \right)^{\lor} }  w_{p+1} & \mbox{for }0\leq p\leq s-1, \\
  		y_p  \xleftarrow{ \lon \left( \overline{ \tilde{\beta}_{k_{u'_p +1}} } \right)^{\lor} }  \cdots  
		\xleftarrow{ \lon \left( \overline{ \tilde{\beta}_{k_{u'_{p+1}}} } \right)^{\lor} } y_{p+1} & \mbox{for }0\leq p\leq s-1;
\end{array}
	\end{align}
here we note that $w_0 = y_0 =w \lon$, and  $\lfloor w_{p} \rfloor = \lfloor y_{p} \rfloor = v_p$, $1 \leq p\leq s$.

Now, let $0\leq p \leq s-1$, and assume that $w_p = y_p$ and $u_p = u'_p$.
Then both of the directed paths in (\ref{surj_induction}) are directed paths from some element in $v_{p+1} W_S$ to $w_p$ in $\QBG (W)$ whose edge labels are increasing and lie in $\Delta^+ \setminus \Delta_S^+$.
Therefore, $w_{p+1} = y_{p+1} \in v_{p+1} W_S$ and $\lon \left( \overline{\tilde{\beta}_{j_i}} \right)^\lor= \lon \left( \overline{\tilde{\beta}_{k_i}}\right)^\lor$ for $u_p+1 \leq i \leq u_{p+1}$, and $u_{p+1} = u_{p+1}'$
by Lemma \ref{8.5} (2);
also,
$d_{j_i} = d_{k_i} = \sigma_p$ for $u_p+1 \leq i \leq u_{p+1}$.
It follows from the equalities
$\tilde{\beta}_{j} = \overline{\tilde{\beta}_j} + (1 -d_j)\langle \lambda_- , \overline{\tilde{\beta}_j} \rangle \widetilde{\delta}$, 
$1 \leq j \leq L$,
that
$\tilde{\beta}_{j_i} = \tilde{\beta}_{k_i}$, $u_p+1 \leq i \leq u_{p+1}$, and hence $j_i = k_i$, $u_p+1 \leq i \leq u_{p+1}$.
Thus, by induction on $p$, 
we deduce that $u_p = u_p'$ for all $0 \leq p \leq s$, and that $j_i = k_i$ for all $u_0+1 \leq i \leq u_{s}$.
Consequently, we obtain $r = u_{s}= u_{s}' =r'$, and hence 
$J= \{ j_1 , \ldots , j_r \} = \{ k_1 , \ldots , k_{r'} \}=K$. 
This proves the lemma.
\end{proof}

\begin{lem}\label{surjective}
		The map $\Xi_w :\QB(w; \tra) \rightarrow \QLS(\lambda)$
is surjective.
\end{lem}

\begin{proof}
Take an arbitrary element
$
		\eta =(y_{1} , \ldots , y_{s} ; {\tau}_{0},\ldots, {\tau}_{s})  \in \QLS(\lambda)
$;
we set $y_0=\lfloor w \lon \rfloor \in W^S$.
We define elements $v_{p} \in W$, $0 \leq p \leq s$, by:
$v_0 =w \lon $,
and
$v_{p} = \min(y_{p}W_S , \leq_{v_{p-1}})$  for $1\leq p \leq s$.

Because there exists a  directed path in $\QBG_{\tau_p \lambda} (W^S)$ from $y_{p+1}$ to $y_{p}$ for $1\leq p \leq s-1$,
it follows from Lemma \ref{8.5}\,(2),\,(3) that
there exists a unique  directed path 
\begin{equation}\label{path3}
v_p \xleftarrow{-\lon \gamma_{p,1}} \cdots \xleftarrow{-\lon \gamma_{p,t_{p}}} v_{p+1}
\end{equation}
in $\QBG_{\tau_p \lambda} (W)$ from $v_{p+1}$ to $v_p$ whose edge labels $-\lon\gamma_{p, t_p}, \ldots, -\lon \gamma_{p, 1}$ are increasing 
in the reflection order $\prec$,
and lie in $\Delta^+ \setminus \Delta_S^+$
for $1\leq p \leq s-1$.
We remark that this is also the case for $p=0$ since
$\tau_0=0$;
if $y_1=y_0=\lfloor w \lon \rfloor$, then we set $t_0=0$.
Multiplying this directed path on the right by $\lon$, 
we get the following directed  paths by Lemma \ref{involution}:
	\begin{equation}\label{eq:recov}
		v_{p,  0} \eqdef v_{p}\lon  
		\xrightarrow{ \gamma_{p, 1}}  v_{p,  1} 
		\xrightarrow{ \gamma_{p, 2}} 
		\cdots \xrightarrow{ \gamma_{p, t_p}} v_{p+1}\lon 
		\eqdef v_{p,  t_p},
\ \
0 \leq p \leq s-1.
	\end{equation}
Concatenating these paths for $0 \leq p\leq s-1$, we obtain the following directed  path
\begin{equation*}
		w =v_{0,0} \xrightarrow{ \gamma_{0, 1} } \cdots  
		\xrightarrow{ \gamma_{0, t_0} }v_{0,  t_0}=v_{1,  0}
		\xrightarrow{ \gamma_{1, 1} } 
		\cdots 
 \xrightarrow{ \gamma_{1, t_1} }v_{1, t_1} = v_{2,  0}  \xrightarrow{ \gamma_{2, 1} }  
\cdots \xrightarrow{ \gamma_{s-1, t_{s-1}} }  v_{s-1,  t_{s-1}}  
\end{equation*}
	in $\QBG (W)$. 
	Now, for $0\leq p \leq s-1$ and $1 \leq m \leq t_p$,
	we set $d_{p,m}  \eqdef \tau_{p} \in \mathbb{Q}\cap [0,1)$, 
	$a_{p,m} \eqdef (1-d_{p,m})\langle \lambda_- , -{\gamma}^{\lor}_{p,m}   \rangle$, 
	and 
	$\widetilde{\gamma}_{p,m}  \eqdef a_{p,m} \widetilde{\delta} - \gamma_{p,m}^\lor$.
	It follows from (\ref{B}) that $\widetilde{\gamma}_{p,m}   \in 
\ptrr \cap \tra\inv \ntrr$.


\begin{nclaim}
\begin{enu}
\item
We have
\begin{equation*}
\widetilde{\gamma}_{0,1} \prec'  \cdots \prec' \widetilde{\gamma}_{0,t_0} 
\prec'  
\widetilde{\gamma}_{1,1} \prec' \cdots \prec' \widetilde{\gamma}_{s-1,t_{s-1}},
\end{equation*}
where $\prec'$ denotes the total order on $\ptrr \cap \tra\inv \ntrr$
introduced in Section $5.2${\rm;}
hence
we can choose 
$J' = \{ j'_1 , \ldots , j'_{r'} \} \subset \{ 1 , \ldots , L \}$
such that
\begin{equation*}
\left(\tilde{\beta}_{j'_1} , \ldots , \tilde{\beta}_{j'_{r'}} \right)
=
\left(  
\widetilde{\gamma}_{0,1} ,  \ldots , \widetilde{\gamma}_{0, t_0} ,
\widetilde{\gamma}_{1,1} , \ldots , \widetilde{\gamma}_{s-1,t_{s-1}}
\right) .
\end{equation*}

\item
Let $1 \leq k \leq r'$, and take $0 \leq p \leq s-1$ and $1 \leq m \leq t_p$ such that
\begin{equation*}
\left( \tilde{\beta}_{j'_1} \prec' \cdots \prec' \tilde{\beta}_{j'_{k}} \right)
=
\left(  \widetilde{\gamma}_{0,1} \prec'  \cdots \prec'  \widetilde{\gamma}_{p,m} 
 \right).
\end{equation*}
Then, 
$\dr(z_{k}) = v_{p,m}$.
Moreover, 
$\dr(z_{k-1})\xrightarrow{-\left(\overline{\tilde{\beta}_{j'_k} } \right)^\lor} \dr(z_k)$ is a directed edge of $\QBG (W)$.
\end{enu}
\end{nclaim}

\noindent $Proof \ of \ Claim \ 1.$
(1)
It suffices to show the following:

(i)
for $0 \leq p \leq s-1$ and $1\leq m < t_p$,
we have
$  \widetilde{\gamma}_{p,m}\prec' \widetilde{\gamma}_{p,m+1}$;

(ii)
for $0 \leq p \leq s-2$,
we have
$\widetilde{\gamma}_{p,t_p} \prec' \widetilde{\gamma}_{p+1,1}$.

(i)
Because
$\frac{\langle \lambda_- , -\gamma_{p,m}^{\lor}\rangle - a_{p,m}}{\langle \lambda_- , -\gamma_{p,m}^{\lor}\rangle}=d_{p,m}$ and
$\frac{\langle \lambda_- , -\gamma_{p,m+1}^{\lor}\rangle - a_{p,m+1}}{\langle \lambda_- , -\gamma_{p,m+1}^{\lor}\rangle}=d_{p,m+1}$,
we have
\begin{align*}
\Phi (\widetilde{\gamma}_{p,m}) &=
 (d_{p,m}, - \lon \gamma_{p,m} ), \\
\Phi (\widetilde{\gamma}_{p,m+1}) &=
 (d_{p,m+1}, - \lon \gamma_{p,m+1} ).
\end{align*}
Therefore, the first component of $ \Phi (\widetilde{\gamma}_{p,m})$ is equal to that of $\Phi (\widetilde{\gamma}_{p,m+1})  $
since $d_{p,m} =1 - \tau_p =d_{p,m+1}$.
Since $- \lon \gamma_{p,m} \succ -\lon \gamma_{p,m+1}$,
we deduce that $ \Phi (\widetilde{\gamma}_{p,m}) < \Phi (\widetilde{\gamma}_{p,m+1})  $.
This implies that $  \widetilde{\gamma}_{p,m}\prec' \widetilde{\gamma}_{p,m+1}$.

(ii)
The proof of (ii) is similar to that of (i).
The first components of $\Phi (\widetilde{\gamma}_{p,t_p}) $ and $\Phi (\widetilde{\gamma}_{p+1,1}) $ are $d_{p,t_p}$ and $d_{p+1,1}$, respectively.
Since $d_{p,t_p}= \tau_p < \tau_{p+1}= d_{p+1,1}$,
we have $\Phi (\widetilde{\gamma}_{p,t_p}) < \Phi (\widetilde{\gamma}_{p+1,1}) $.
This implies that
$\widetilde{\gamma}_{p,t_p} \prec' \widetilde{\gamma}_{p+1,1}$.

(2)
We proceed 
by induction on $k$.
If $\tilde{\beta}_{j'_1} = \widetilde{\gamma}_{0,1}$, i.e., $y_1 \neq \lfloor w \lon \rfloor$,
then
we have
$\dr (z_1) = \dr (z_0) s_{-\overline{\tilde{\beta}_{j'_1}}} =v_{0,0}  s_{\gamma_{0,1}} = v_{0,1}$,
since $\dr (z_0) = \dr(w \tra)= w = v_{0,0}$.
If $\tilde{\beta}_{j'_1} = \widetilde{\gamma}_{1,1}$, i.e., $y_1 = \lfloor w \lon \rfloor$ and $t_0 =0$,
then
we have
$\dr (z_1) = \dr (z_0) s_{-\overline{\tilde{\beta}_{j'_1}}} = w s_{\gamma_{1,1}} =
v_{1,0} s_{\gamma_{1,1}} = v_{1,1}$,
since $\dr (z_0) = \dr(w \tra) = w = v_{1,0}$.
Hence the assertion holds in the case $k=1$.

Assume that $\dr(z_{k-1}) = v_{p,m-1}$
for $0 \lneqq m \leq t_p$;
here we remark that $v_{p,m}$ is the successor of $v_{p,m-1}$ in the  directed path \eqref{eq:recov}.
Therefore, we see that
\begin{align*}
\dr(z_{k}) 
&= \dr(z_{k-1}) s_{-\overline{ \tilde{\beta}_{j'_k} } }
= v_{p,m-1}s_{\gamma_{p,m}}
=v_{p,m} \\
&\ \hspace{55mm}
\mbox{(since $v_{p,m-1}\xrightarrow{\gamma_{p,m}}v_{p,m}$ is a directed edge in \eqref{eq:recov})}.
\end{align*}
Also, since \eqref{eq:recov} is a directed path in $\QBG (W)$, 
$ v_{p,m-1} = \dr(z_{k-1})\xrightarrow{-\left(\overline{ \tilde{\beta}_{j'_k} } \right)^\lor} \dr(z_k) = v_{p,m} $ is a directed edge of $\QBG (W)$.
\bqed

Since $J' = \{ j'_1 , \ldots , j'_{r'} \} \subset \{ 1 , \ldots , L \}$,
we can define an element
$
p_{J'} $
by\\
$ \left( w \tra = z_0 ,  z_{1} , \ldots , z_{r'}
; \tilde{\beta}_{j'_1} , \tilde{\beta}_{j'_2}, \ldots , \tilde{\beta}_{j'_{r'}} \right) 
$,
where
$z_{0}=w \tra$, $z_{k}=z_{k-1} s_{\tilde{\beta}_{j'_{k}}}$ for $1\leq k \leq r'$;
it follows from Remark \ref{affinereflectionorder} and Claim 1 that
 $p_{J'} \in \QB(w; \tra)$.

\begin{nclaim}
$\Xi_w (p_{J'})=\eta$.
\end{nclaim}

\noindent $Proof \ of \ Claim \ 2.$
In the following description of $ p_{J'}$,
we employ the notation  $u_p$, $\sigma_p$, $w'_p$, and $w_p$, 
$ 0 \leq p \leq s$,
used  in the definition of $\Xi_w (p_{J})$.

For $1 \leq k \leq r'$, if we set $\tilde{\beta}_{j'_k} \eqdef \widetilde{\gamma}_{p,m}$,
then we have
\begin{equation*}
d_{j'_k}=
1 + \frac{\deg(\tilde{\beta}_{j'_k})}{\langle \lambda_- , -\overline{\tilde{\beta}_{j'_k} }^\lor \rangle}
=
1 + 
\frac{\deg(\widetilde{\gamma}_{p,m})}{\langle \lambda_- , -\overline{\widetilde{\gamma}_{p,m}}^\lor \rangle}
=
1 + \frac{a_{p,m}}{\langle \lambda_- , \gamma_{p,m}^\lor \rangle}
=
d_{p,m}
.
\end{equation*}
Therefore,
the sequence (\ref{2.16}) determined by $p_{J'}$ is 
\begin{equation}\label{2.17}
		\underbrace{ 0 = d_{0,1} = \cdots =d_{0,t_0} }_{=\tau_0}
		< \underbrace{d_{1,1} = \cdots =d_{1, t_1}}_{=\tau_1}
		 < \cdots <
		\underbrace{ d_{s-1,1} = \cdots =d_{s-1,t_{s-1}} }_{=\tau_{s-1}} <1 = \tau_{s} = \sigma_s.
	\end{equation}
Because the sequence (\ref{2.17}) of rational numbers is just the sequence (\ref{2.16}) for 
$\Theta ( \eta ) = p_{J'}$,
we deduce that 
$u_{p+1} - u_p = t_p$ for $0 \leq p \leq s-1$,
$\tilde{\beta}_{j'_{u_p + k}} = \widetilde{\gamma}_{p,k}$ for $0 \leq p \leq s-1$, $1 \leq k \leq u_{p+1} - u_p$,
and
$\sigma_p = \tau_{p}$ for $0 \leq p \leq s$.
From these, we see that
\begin{align*}
w'_p = \dr(z_{u_p})=
v_{u_{p-1}, t_{p-1}}
=v_{p,0} 
 \hspace{5mm}
\mbox{(since $v_{p,m-1}\xrightarrow{\gamma_{p,m}}v_{p,m}$ is a directed edge in \eqref{eq:recov})}.
\end{align*}
 and $w_p = v_{p,0}\lon =v_{p}$.
Since $\lfloor w_p \rfloor = \lfloor v_{p} \rfloor = y_{p}$,
we conclude that \\
$\Xi_w (p_{J'}) = (\lfloor w_1 \rfloor ,\ldots , \lfloor w_s \rfloor ; \sigma_0, \ldots , \sigma_s)=(y_1 ,\ldots, y_s ; \tau_0 ,\ldots , \tau_s)= \eta$, as desired.
\bqed

This completes the proof of the lemma.
\end{proof}

By Lemmas \ref{injective} and \ref{surjective}, we obtain the following proposition.
\begin{prop}\label{bijective}
The map $\Xi_w$ is bijective.
In particular, the cardinality of the set $\QB(w; \tra)$ is independent of $w \in W$.
\end{prop}

\begin{rem}
In the proof of Lemma \ref{surjective},
we showed that for $\eta \in \QLS(\lambda)$, there exists a unique $J \subset \{ 1, \ldots , L\}$
such that $p_J \in \QB(w, \tra)$ and $\Xi_w(p_J)=\eta$;
this assignment gives
 the inverse map $\Xi\inv_w:\QLS(\lambda) \rightarrow \QB(w; \tra)$ of $\Xi_w$.
\end{rem}

Recall from \eqref{eq:dfn_deg} and \eqref{eq:dfn_wt} that $\deg(\beta)$ is defined by: $\beta= \overline{\beta} + \deg(\beta)\widetilde{\delta}$ for $\beta \in \mathfrak{h}_\mathbb{R} \oplus \mathbb{R}\widetilde{\delta}$, and $\wt(u) \in P$ and $\dr(u) \in W$ are defined by: $u=t(\wt (u))\dr(u)$ for $u\in W_\ext = t(P)\rtimes W$.
\begin{prop}\label{qls_qb}
The bijection $\Xi_w:\QB(w; \tra) \rightarrow \QLS(\lambda)$ has the following properties{\rm:}
\begin{enu}
\item
		$\wt (\ed (p_{J})) = \wt(\Xi_w(p_{J}) )${\rm;}
\item
		$\deg (\qwt(p_{J})) = \Dgw (\Xi_w(p_{J}))$.
\end{enu}
\end{prop}
\setcounter{nclaim}{0} 

\begin{proof}
We proceed by induction on $ \#J$.

If $J=\emptyset$, it is obvious that $\deg (\qwt(p_{J})) = \Dgw (\Xi_w (p_{J}))=0$ and 
${\wt}(\ed (p_{J})) = \wt(\Xi_w(p_{J}) )= w \lambda_-$,
since
$\Xi_w(p_{J})=(\lfloor w v(\lambda_-)\rfloor ; 0,1)$.
	
Let	$J=\{ j_1 <j_2 <\cdots <j_r\}$, and set
	$K \eqdef J\setminus \{ j_r \}$;
	assume that $\Xi_w (p_{K})$ is of the form: 
$\Xi_w (p_{K}) =
  (\lfloor w_1 \rfloor  ,\ldots , \lfloor w_s \rfloor ; \sigma_0 ,\ldots, \sigma_s)  $.
In the following, we employ the notation $w_p$, $0\leq p \leq s$, used in the definition of the map $\Xi_w$. 
Note that
$\dr(\ed(p_K))=w_s \lon$
by the definition of $\Xi_w$.
Also, note that
if $d_{j_{r}}=d_{j_{r-1}}={\sigma}_{s-1}$, 
then $\{ d_{j_1} \leq \cdots \leq d_{j_{r-1}}\leq d_{j_r} \}
=
\{ d_{j_1} \leq \cdots \leq d_{j_{r-1}} \} $,
and
that
if $d_{j_{r}}>d_{j_{r-1}}={\sigma}_{s-1}$,
then
$\{ d_{j_1} \leq \cdots \leq d_{j_{r-1}}\leq d_{j_r} \}
= 
\{ d_{j_1} \leq \cdots \leq d_{j_{r-1}}< d_{j_r} \}
$.
From these, we deduce that
	\begin{align*}
	\Xi_w (p_{J})=
\begin{cases}
(\lfloor w_1 \rfloor ,\ldots,  \lfloor w_{s-1}\rfloor , \lfloor w_{s} s_{\lon \overline {\tilde{\beta}_{j_r} }  }  \rfloor ; \sigma_0, \ldots, \sigma_{s-1}, \sigma_s)
& \mbox{if } d_{j_{r}}=d_{j_{r-1}} ={\sigma}_{s-1}, \\
(\lfloor w_1 \rfloor ,\ldots , \lfloor w_{s-1} \rfloor ,  \lfloor w_{s} \rfloor ,\lfloor  w_{s}s_{\lon \overline {\tilde{\beta}_{j_r }} }  \rfloor ; \sigma_0, \ldots, \sigma_{s-1}, d_{j_r}, \sigma_{s})
& \mbox{if } d_{j_{r}}>d_{j_{r-1}} ={\sigma}_{s-1}.
\end{cases}
	\end{align*}

For the induction step, it suffices to show the following claims.

\begin{nclaim}
\
\begin{enu}
\item
We have
	\begin{equation*}
		\wt ( \Xi_w (p_{J}) ) 
		= \wt ( \Xi_w (p_{K}) )+ 
		a_{j_{r}}w_{s}\lon \left( -\overline {\tilde{\beta}_{j_r } } \right)^\lor.
	\end{equation*}

\item
We have
\begin{equation*}
		\Dgw ( \Xi_w (p_{J}) ) 
		= \Dgw ( \Xi_w (p_{K}) )+ \chi \deg(\tilde{\beta}_{j_r} ),
	\end{equation*}
where
$\chi \eqdef 0$ {\rm(}resp., $\chi \eqdef 1${\rm)} 
if $w_{s}s_{\lon \overline {\tilde{\beta}_{j_r } }  } \rightarrow w_{s}$ is a Bruhat {\rm(}resp., quantum{\rm)} edge of $\QBG(W)$.
\end{enu}
\end{nclaim}

\begin{nclaim}
\
\begin{enu}
\item
We have
	\begin{equation*}
		\wt ( \ed (p_{J}) ) 
		= \wt ( \ed (p_{K}) )+ 
		a_{j_{r}}w_{s}\lon \left( -\overline {\tilde{\beta}_{j_r } } \right)^\lor.
	\end{equation*}

\item
We have
	\begin{equation*}
		\deg ( \qwt (p_{J}) )
		= \deg ( \qwt (p_{K}) ) + \chi \deg(\tilde{\beta}_{j_r}).
	\end{equation*}
\end{enu}
\end{nclaim}

The proofs of Claims 1 and 2 are the same as those of 
 Claims 1 and 2 in \cite[Proposition 3.3.6]{NNS}, respectively. 
This proves the proposition.
\end{proof}

\begin{proof}[Proof of Theorem $\ref{thm:graded_character}$]
It follows from Propositions \ref{bijective} and \ref{qls_qb} that
\begin{align*}
\overline{C^{t(\lon \lambda)}_w}
&=
\sum_{p_J \in \QB(w; t(\lon \lambda))}q^{-\deg(\qwt(p_J))}e^{\wt(\ed(p_J))}\\
&=
\sum_{p_J \in \QB(w; t(\lon \lambda))}q^{-\Dgw(\Xi_w(p_J))}e^{\wt(\Xi_w(p_J))}\\
&=
\sum_{\eta \in \QLS(\lambda)}q^{-\Dgw(\eta)}e^{\wt(\eta)}\\
&=
\gch^{w \lon}\QLS(\lambda),
\end{align*}
as desired.
\end{proof}

\end{document}